\documentclass[11pt,twoside,reqno]{amsart}

\usepackage{microtype}
\usepackage{cite}
\usepackage[OT1]{fontenc}
\usepackage{type1cm}
\usepackage{amssymb}
\usepackage{comment}
\usepackage{xcolor}

\usepackage{geometry}
\usepackage{hyperref}
\hypersetup{
  colorlinks=true,
  linkcolor=black,
  anchorcolor=black,
  citecolor=black,
  filecolor=black,      
  menucolor=red,
  runcolor=black,
  urlcolor=black,
}

\numberwithin{equation}{section}

\theoremstyle{plain}
\newtheorem{theorem}{Theorem}[section]

\newtheorem{proposition}[theorem]{Proposition}

\newtheorem{lemma}[theorem]{Lemma}

\theoremstyle{remark}

\newtheorem*{ack}{Acknowledgement}

\theoremstyle{definition}

\newcommand{\LL}{\mathcal{L}}

\newcommand{\MM}{\mathcal{M}}

\newcommand{\CC}{\mathcal{C}}

\newcommand{\R}{\mathbb{R}}
\newcommand{\RP}{\mathbb{RP}^1}

\newcommand{\N}{\mathbb{N}}

\newcommand{\iii}{\mathtt{i}}
\newcommand{\jjj}{\mathtt{j}}
\newcommand{\kkk}{\mathtt{k}}

\newcommand{\fii}{\varphi}
\newcommand{\roo}{\varrho}

\newcommand{\A}{\mathsf{A}}

\renewcommand{\ge}{\geqslant}
\renewcommand{\le}{\leqslant}
\renewcommand{\geq}{\geqslant}
\renewcommand{\leq}{\leqslant}

\DeclareMathOperator{\udimm}{\overline{dim}_M}

\DeclareMathOperator{\dimh}{dim_H}

\DeclareMathOperator{\dimaff}{dim_{aff}}

\DeclareMathOperator{\linspan}{span}

\DeclareMathOperator{\diam}{diam}
\DeclareMathOperator{\proj}{proj}

\DeclareMathOperator{\rank}{rank}

\DeclareMathOperator{\im}{im}

\begin{document}

\title{Non-invertible planar self-affine sets}

\author{Antti K\"aenm\"aki}
\address[Antti K\"aenm\"aki]
        {Research Unit of Mathematical Sciences \\ 
         P.O.\ Box 8000 \\ 
         FI-90014 University of Oulu \\ 
         Finland}
\email{antti.kaenmaki@oulu.fi}

\author{Petteri Nissinen}
\address[Petteri Nissinen]
        {Department of Physics and Mathematics \\
         University of Eastern Finland \\
         P.O.\ Box 111 \\
         FI-80101 Joensuu \\
         Finland}
\email{pettern@student.uef.fi}

\subjclass[2000]{Primary 28A80; Secondary 37C45, 37D35.}
\keywords{Self-affine set, Hausdorff dimension}
\date{\today}

\begin{abstract}
  We compare the dimension of a non-invertible self-affine set to the dimension of the respective invertible self-affine set. In particular, for generic planar self-affine sets, we show that the dimensions coincide when they are large and differ when they are small. Our study relies on thermodynamic formalism where, for dominated and irreducible matrices, we completely characterize the behavior of the pressures.
\end{abstract}

\maketitle

\section{Introduction}

Let $J$ be a finite set and $(A_i+v_i)_{i \in J}$ a tuple of contractive affine self-maps on $\R^2$, where we have written $A+v$ to denote the affine map $x \mapsto Ax+v$ defined on $\R^2$ for all matrices $A \in M_2(\R)$ and translation vectors $v \in \R^2$. If the affine maps $A_i+v_i$ do not have a common fixed point, then we call such a tuple an \emph{affine iterated function system}. We also write $f_i = A_i+v_i$ for all $i \in J$ and note that the associated tuple of matrices $(A_i)_{i \in J}$ is an element of $M_2(\R)^J$.

A classical result of Hutchinson \cite{Hutchinson1981} shows that for each affine iterated function system $(f_i)_{i \in J}$ there exists a unique non-empty compact set $X' \subset \R^2$, called the \emph{self-affine set}, such that
\begin{equation} \label{eq:self-affine-set-def}
  X' = \bigcup_{i \in J} f_i(X').
\end{equation}
In this article, if $I = \{i \in J : A_i \text{ is invertible}\}$ is non-empty, then the self-affine set $X \subset X'$ associated to $(f_i)_{i \in I}$ is called \emph{invertible}, and if $J \setminus I$ is non-empty, then the self-affine set $X'$ associated to $(f_i)_{i \in J}$ is called \emph{non-invertible}. B\'ar\'any, Hochman, and Rapaport \cite{BHR} and Hochman and Rapaport \cite{HochmanRapaport2021} have recently shown that the Hausdorff dimension reaches a natural upper bound, the affinity dimension, on a large deterministic class of invertible self-affine sets.

In our main result, Theorem \ref{thm:main} below, part (1) shows that generically under a separation condition the dimensions of $X'$ and $X$ agree when they are at least $1$. Furthermore, if the dimension of $X$ is strictly less than $1$, then part (2) demonstrates that generically the dimensions of $X'$ and $X$ are distinct. Regarding part (3), let us first recall that Marstrand's projection theorem \cite{Marstrand1954} gives $\dimh(\proj_{V}(X')) = \min\{1,\dimh(X')\}$ for Lebesgue almost all $V \in \RP$. Although the equality holds for generic $V$, it is often difficult to say whether a particular $V$ satisfies it. The purpose of part (3) is to verify that the orthogonal complement of the kernel of one of the rank one matrices is such a direction.

The precise definitions of the assumptions used in the theorem will be given in coming sections.

\begin{theorem} \label{thm:main}
  Suppose that $X'$ and $X$ are the planar self-affine sets associated to affine iterated function systems $(A_i+v_i)_{i \in J}$ and $(A_i+v_i)_{i \in I}$ such that $A_i \in GL_2(\R)$ for all $i \in I \subset J$, respectively.
  \begin{enumerate}
    \item If $(A_i)_{i \in I}$ is strictly affine and strongly irreducible such that $\dimaff((A_i)_{i \in I}) \ge 1$ and $X$ satisfies the strong open set condition, then
    \begin{align*}
      \udimm(X') &= \dimh(X), \\ 
      \dimh(\proj_V(X')) &= 1
    \end{align*}
    for all $V \in \RP$.
    \item If $(A_i)_{i \in J}$ is dominated or irreducible such that $\max_{i \in J} \|A_i\| < \frac12$, contains a rank one matrix, and $\dimaff((A_i)_{i \in I}) < 1$, then
    \begin{equation*}
      \dimh(X_{\mathsf{v}}') > \udimm(X_{\mathsf{v}})
    \end{equation*}
    for $\LL^{2\# J}$-almost all translation vectors $\mathsf{v} = (v_i)_{i \in J} \in (\R^2)^{\# J}$.
    \item If $(A_i)_{i \in J}$ contains a rank one matrix, $(A_i)_{i \in I}$ is strictly affine and strongly irreducible such that $\dimaff((A_i)_{i \in I}) < 1$, and $X$ satisfies the strong open set condition, then there exists a rank one matrix $A$ in $\A$ such that
    \begin{equation*}
      \dimh(X') = \dimh(\proj_{\ker(A)^\bot}(X')) \le 1.
    \end{equation*}
  \end{enumerate}
\end{theorem}

We remark that B\'ar\'any and K\"ortv\'elyesi \cite{BaranyKort2024} have recently continued the above study. They have demonstrated that if the affinity dimension is strictly less than one, then there exist two large parameter sets for the defining matrices so that in the first one, the Hausdorff dimension of the non-invertible self-affine set equals the affinity dimension, and in the second one, the Hausdorff dimension is strictly smaller than the affinity dimension. This observation proposes that determining the Hausdorff dimension in this situation requires a better understanding of the geometry.

The remainder of the paper is organized as follows. In Section \ref{sec:matrices}, we compare the behavior of the pressures and study the continuity. In particular, for dominated and irreducible matrices, we completely characterize the continuity of the pressure in the non-invertible case. In Section \ref{sec:dim-results}, we uncover how the study of non-invertible self-affine sets is connected to the theory of sub-self-affine and inhomogeneous self-affine sets, and prove the main result.

\section{Products of matrices} \label{sec:matrices}

\subsection{Rank one matrices}
We denote the collection of all $2 \times 2$ matrices with real entries by $M_2(\R)$, the general linear group of degree $2$ over $\R$ by $GL_2(\R) \subset M_2(\R)$, and the orthogonal group in dimension $2$ over $\R$ by $O_2(\R) \subset GL_2(\R)$. A matrix $A \in GL_2(\R)$ is called \emph{proximal} if it has two real eigenvalues with different absolute values. If $A \in M_2(\R)$, then the \emph{singular values} of $A$ are defined to be the non-negative square roots of the eigenvalues of the positive-semidefinite matrix ${A^\top}A$ and are denoted by $\alpha_1(A)$ and $\alpha_2(A)$ in non-increasing order. Recall that the rank of $A$ is the number of non-zero singular values of $A$. The identities $\alpha_1(A)=\|A\|$ and $\alpha_1(A)\alpha_2(A) = |\det(A)|$ for all $A \in M_2(\R)$ are standard, as is the identity $\alpha_2(A)=\|A^{-1}\|^{-1}$ in the case where $A$ is invertible. For each $A \in M_2(\R)$ and $s \geq 0$ we define the \emph{singular value function} by setting
\begin{equation*}
  \varphi^s(A)=
  \begin{cases}
    \alpha_1(A)^s, &\text{if } 0 \le s \le 1, \\ 
    \alpha_1(A)\alpha_2(A)^{s-1}, &\text{if } 1 < s \le 2, \\
    |\det(A)|^{s/2}, &\text{if } 2 < s < \infty,
  \end{cases}
\end{equation*}
where we interpret $0^0 = 1$. The value $\fii^s(A)$ represents a measurement of the $s$-dimensional volume of the image of the Euclidean unit ball under $A$. Since $\alpha_1(A)\alpha_2(A)^{s-1} = \alpha_1(A)^{2-s}|\det(A)|^{s-1}$ for all $1 < s \le 2$, the inequality $\varphi^s(AB) \leq \varphi^s(A)\varphi^s(B)$ is valid for all $s \ge 0$. In other words, the singular value function is sub-multiplicative.

Note that if $A \in M_2(\R)$ has rank one, then $\fii^s(A) = 0$ for all $s>1$. Recalling that $A$ has rank zero if and only if $A$ is the zero matrix, we see that $\fii^s(A)=0$ for all $s > 0$. Let us next recall that rank one matrices are projections. Let $\RP$ be the real projective line, that is, the set of all lines through the origin in $\R^2$. If $V,W \in \RP$, then the \emph{projection} $\proj_V^W \colon \R^2 \to V$ is the linear map such that $\proj_V^W|_V=\mathrm{Id}|_V$ and $\ker(\proj_V^W)=W$. Furthermore, the \emph{orthogonal projection} $\proj_V^{V^\bot}$ onto the subspace $V$ is denoted by $\proj_V$. The following lemma is well-known. But, as the proof is short, we provide the reader with full details.

\begin{lemma} \label{thm:rank-one1}
  A matrix $A \in M_2(\R)$ has rank one if and only if there exist $v,w \in \R^2 \setminus \{(0,0)\}$ such that $A = vw^\top$. In this case,
  \begin{equation*}
    A =
    \begin{cases}
      \langle v,w \rangle\proj_{\im(A)}^{\ker(A)}, &\text{if $A$ is not nilpotent}, \\
      |v||w|R\proj_{\ker(A)^\perp}, &\text{if $A$ is nilpotent},
    \end{cases}
  \end{equation*}
  where $R \in O_2(\R)$ is a rotation by an angle $\pi/2$. In particular, $A(X)$ is bi-Lipschitz equivalent to $\proj_{\ker(A)^\bot}(X)$ for all $X \subset \R^2$.
\end{lemma}

\begin{proof}
  Let us first prove the characterization of rank one matrices. If $A = vw^\top$ for some $v,w \in \R^2 \setminus \{(0,0)\}$, then $Ax = vw^\top x = \langle w,x \rangle v$ for all $x \in \R^2$. Therefore, $A$ maps every $x$ to a scalar multiple of $v$, $\rank(A)=1$, and $\im(A) = \linspan(v)$. If $x \in \linspan(w)^\bot$, then $Ax = vw^\top x = \langle w,x \rangle v = 0$ and $\ker(A)=\linspan(w)^\bot$. Conversely, if $\rank(A)=1$, then there is $v \in \R^2 \setminus \{(0,0)\}$ such that $Ax$ is a scalar multiple of $v$ for all $x \in \R^2$. In particular, this is true when $x = (1,0)$ and $x = (0,1)$. That is, there are $w_1,w_2 \in \R \setminus \{0\}$ such that $A(1,0)=w_1v$ and $A(0,1)=w_2v$. In other words,
  \begin{equation*}
    A =
    \begin{pmatrix}
      w_1v_1 & w_2v_1 \\ 
      w_1v_2 & w_2v_2
    \end{pmatrix}
    = vw^\top,
  \end{equation*}
  where $w = (w_1,w_2) \in \R^2 \setminus \{(0,0)\}$.

  Let us then show that a rank one matrix $A$ is a projection. If $A$ is not nilpotent, then $\linspan(v) = \im(A) \ne \ker(A) = \linspan(w)^\bot$. Since $Ax = vw^\top x = \langle x,w \rangle v$ and it is easy to see that
  \begin{equation*}
    \proj_{\im(A)}^{\ker(A)}(x) = \frac{\langle x,w \rangle}{\langle v,w \rangle} v = \frac{1}{\langle v,w \rangle} Ax
  \end{equation*}
  for all $x \in \R^2$, we have shown the first case. If $A$ is nilpotent, then $\linspan(v) = \im(A) = \ker(A) = \linspan(w)^\bot$. Since $Rw/|w| = v/|v|$, where $R \in O_2(\R)$ is a rotation by an angle $\pi/2$, we have
  \begin{equation*}
    \proj_{\ker(A)^\bot}(x) = \proj_{\linspan(w)}(x) = \frac{\langle x,w \rangle}{|w|^2} w = \frac{\langle x,w \rangle}{|v||w|} R^{-1}v
  \end{equation*}
  and hence,
  \begin{equation*}
    R\proj_{\ker(A)^\bot}(x) = \frac{\langle x,w \rangle}{|v||w|} v = \frac{1}{|v||w|} Ax
  \end{equation*}
  as claimed.

  Since the last claim follows immediately from the the fact that a rank one matrix is a projection, we have finished the proof.
\end{proof}

\subsection{Pressure}
Let $J$ be a finite set and $\A = (A_i)_{i \in J} \in M_2(\R)^J$ be a tuple of matrices. We say that $\A$ is \emph{irreducible} if there does not exist $V \in \RP$ such that $A_iV \subset V$ for all $i \in J$; otherwise $\A$ is \emph{reducible}. Note that the irreducibility is equivalent to the property that the matrices in $\A$ do not have a common eigenvector. Therefore, $\A$ is reducible if and only if the matrices in $\A$ can simultaneously be presented (in some coordinate system) as upper triangular matrices. The tuple $\A$ is \emph{strongly irreducible} if there does not exist a finite set $\mathcal{V} \subset \RP$ such that $A_i\mathcal{V}=\mathcal{V}$ for all $i\in J$.

We call a proper subset $\CC\subset\RP$ a \emph{multicone} if it is a finite union of closed non-trivial projective intervals. We say that $\A$ is \emph{dominated} if each matrix $A_i$ is non-zero and there exists a multicone $\CC\subset\RP$ such that $A_i\CC\subset\CC^o$ for all $i \in J$, where $\CC^o$ is the interior of $\CC$. Conversely, if a multicone $\CC\subset\RP$ satisfies such a condition, then we say that $\CC$ is a \emph{strongly invariant multicone} for $\A$. For example, the first quadrant is strongly invariant for any tuple of positive matrices. Note that a dominated tuple is not necessarily irreducible and vice versa. If $\A \in GL_2(\R)^J$ is dominated and irreducible, then, by \cite[Lemma 2.10]{BaranyKaenmakiYu2021}, $\A$ is strongly irreducible.

We let $J^*$ denote the set of all finite words $\{ \varnothing \} \cup \bigcup_{n \in \N} J^n$, where $\varnothing$ satisfies $\varnothing\iii = \iii\varnothing = \iii$ for all $\iii \in J^*$. For notational convenience, we set $J^0 = \{ \varnothing \}$. The set $J^\N$ is the collection of all infinite words. We define the \emph{left shift} $\sigma \colon J^\N \to J^\N$ by setting $\sigma\iii = i_2i_3\cdots$ for all $\iii = i_1i_2\cdots \in J^\N$. The concatenation of two words $\iii \in J^*$ and $\jjj \in J^* \cup J^\N$ is denoted by $\iii\jjj \in J^* \cup J^\N$ and the length of $\iii \in J^* \cup J^\N$ is denoted by $|\iii|$. If $\jjj \in J^* \cup J^\N$ and $1 \le n < |\jjj|$, then we define $\jjj|_n$ to be the unique word $\iii \in J^n$ for which $\iii\kkk = \jjj$ for some $\kkk \in J^* \cup J^\N$. Write $\iii|_0 = \varnothing$. If $\iii \in J^* \setminus \{\varnothing\}$, then $\iii^- = \iii|_{|\iii|-1}$ is the word obtained from $\iii$ by deleting its last element. Furthermore, if $\iii \in J^n$ for some $n \in \N$, then we set $[\iii] = \{\jjj \in J^\N : \jjj|_n=\iii\}$. The set $[\iii]$ is called a \emph{cylinder set}. We write $A_\iii = A_{i_1} \cdots A_{i_n}$ for all $\iii = i_1 \cdots i_n \in J^n$ and $n \in \N$. We say that $\mathsf{A} \in GL_2(\R)^J$ is \emph{strictly affine} if there is $\iii \in I^*$ such that $A_\iii$ is proximal. Recall that $A \in GL_2(\R)$  is \emph{proximal} if it has two real eigenvalues with different absolute values. By \cite[Corollary 2.4]{BaranyKaenmakiMorris2018}, a dominated tuple in $GL_2(\R)^J$ is strictly affine.

If $\Gamma \subset J^\N$ is a non-empty compact set such that $\sigma(\Gamma) \subset \Gamma$, then we define $\Gamma_n = \{\iii|_n \in J^n : \iii \in \Gamma\}$ and $\Gamma_* = \bigcup_{n \in \N} \Gamma_n$. We keep denoting $(I^\N)_n$ and $(I^\N)_*$ by $I^n$ and $I^*$, respectively, for all $I \subset J$ and $n \in \N$. Given a tuple $\A = (A_i)_{i \in J} \in M_2(\R)^J$ of matrices, we define for each such $\Gamma \subset J^\N$ and $s \ge 0$ the \emph{pressure} by setting
\begin{equation*} 
  P(\Gamma,\A,s) = \lim_{n \to \infty} \frac{1}{n} \log\sum_{\iii \in \Gamma_n} \fii^s(A_\iii) = \inf_{n \in \N} \frac{1}{n} \log\sum_{\iii \in \Gamma_n} \fii^s(A_\iii) \in [-\infty,\infty).
\end{equation*}
The assumption $\sigma(\Gamma) \subset \Gamma$ guarantees that if $\iii \in J^m$ and $\jjj \in J^n$ such that $\iii\jjj \in \Gamma_{m+n}$, then $\iii \in \Gamma_m$ and $\jjj \in \Gamma_n$. Therefore, as the singular value function is sub-multiplicative, the sequence $(\log\sum_{\iii \in \Gamma_n} \fii^s(A_\iii))_{n \in \N}$ is sub-additive and hence, the limit above exists or is $-\infty$ by Fekete's lemma.

Let $\A$ be a tuple of strictly contractive matrices and $\Gamma \subset J^\N$ be a non-empty compact set such that $\sigma(\Gamma) \subset \Gamma$. Since $\fii^s(A_i) \le \fii^t(A_i) \max_{k \in J}\|A_k\|^{(s-t)}$ for all $i \in J$, we see that $P(\Gamma,\A,s) \le P(\Gamma,\A,t) + (s-t) \log\max_{k \in J}\|A_k\|$ for all $s > t \ge 0$. Since $\A$ consists only of strictly contractive matrices, we have $\max_{k \in J}\|A_k\|<1$ and hence, the pressure $P(\Gamma,\A,s)$ is strictly decreasing as a function of $s$ whenever it is finite. Notice also that $P(\Gamma,\A,0) = \lim_{n \to \infty} \frac{1}{n} \log \#\Gamma_n \ge 0$ and $\lim_{s \to \infty} P(\Gamma,\A,s) = -\infty$. In this case, we define the \emph{affinity dimension} by setting
\begin{equation*}
  \dimaff(\Gamma,\A) = \inf\{s \ge 0 : P(\Gamma,\A,s) \le 0\}.
\end{equation*}
Notice that if the pressure $s \mapsto P(\Gamma,\A,s)$ is continuous at $s_0 = \dimaff(\Gamma,\A)$, then $P(\Gamma,\A,s_0) = 0$.

We are interested in the properties of the pressure
\begin{equation*}
  P(\A,s) = P(J^\N,\A,s)
\end{equation*}
as a function of $s$ and the affinity dimension $\dimaff(\A) = \dimaff(J^\N,\A)$. To that end, let us introduce some further notation. Let $I = \{i \in J : A_i \text{ is invertible}\}$. In this case, we trivially have that
\begin{equation*}
  I^\N = \{\iii \in J^\N : A_{\iii|_n} \text{ is invertible for all }n \in \N\}
\end{equation*}
is a compact subset of $J^\N$ and satisfies $\sigma(I^\N) = I^\N$. Therefore, the pressure $P(I^\N,\A,s)$ is well-defined for all $s \ge 0$. We also define
\begin{equation*}
  \Sigma = \{\iii \in J^\N : A_{\iii|_n} \text{ is non-zero for all }n \in \N\}.
\end{equation*}
It is easy to see that $\Sigma$ is a compact subset of $J^\N$ and satisfies $\sigma(\Sigma) \subset \Sigma$. Indeed, if $\jjj \in \sigma(\Sigma)$, then there is $\iii \in \Sigma$ such that $\jjj = \sigma\iii$ and $A_{\iii|_n} \ne 0$ for all $n \in \N$. As clearly $A_{\sigma\iii|_n} \ne 0$ for all $n \in \N$, we see that $\jjj = \sigma\iii \in \Sigma$ as claimed. Hence, also the pressure $P(\Sigma,\A,s)$ is well-defined for all $s \ge 0$. Observe that the inclusion $\sigma(\Sigma) \subset \Sigma$ can be strict: if $J = \{0,1\}$ and
\begin{equation*}
  A_0 =
  \begin{pmatrix}
    0 & 1 \\ 0 & 0
  \end{pmatrix}, \qquad
  A_1 =
  \begin{pmatrix}
    0 & 0 \\ 0 & 1
  \end{pmatrix},
\end{equation*}
then $\Sigma = \{0111\cdots, 111\cdots\}$ and $\sigma(\Sigma) = \{111\cdots\}$.

\begin{lemma} \label{thm:pressure-continuous}
  If $\A = (A_i)_{i \in J} \in M_2(\R)^J$ satisfies $\max_{i \in J} \|A_i\| < 1$, then
  \begin{equation*}
    P(\A,s) =
    \begin{cases}
      \log \# J, &\text{if } s = 0, \\
      P(\Sigma,\A,s), &\text{if } 0 < s \le 1, \\ 
      P(I^\N,\A,s), &\text{if } 1 < s < \infty.
    \end{cases}
  \end{equation*}
  Furthermore, the function $s \mapsto P(\A,s)$ is strictly decreasing on $[0,\infty)$, continuous on $(0,1)$, and uniformly continuous on $(1,\infty)$ whenever it is finite.
\end{lemma}

\begin{proof}
  Recall first that $\fii^s(A) = \alpha_1(A)^s = \|A\|^s$ for all $0 \le s \le 1$. Therefore, as we interpreted $0^0=1$, we have
  \begin{equation*}
    P(\A,0) = \lim_{n \to \infty} \frac{1}{n} \log \sum_{\iii \in J^n} \|A_\iii\|^0 = \log \# J.
  \end{equation*}
  Since $\alpha_1(A)>0$ if and only if $A \in M_2(\R)$ is non-zero, we see that for each $0 < s \le 1$ the singular value function satisfies $\fii^s(A_\iii) = \|A_\iii\|^s > 0$ if and only if $\iii \in \Sigma_*$. Therefore, $P(\A,s) = P(\Sigma,\A,s)$ for all $0 < s \le 1$. Furthermore, since $\alpha_2(A)>0$ if and only if $A \in GL_2(\R)$, we have that for every $1<s<\infty$ the singular value function satisfies $\fii^s(A_\iii)>0$ if and only if $\iii \in I^*$. This shows $P(\A,s) = P(I^\N,\A,s)$ for all $1<s<\infty$. The function $s \mapsto P(\A,s)$ has already seen strictly decreasing. The continuity on $(0,1)$ follows from \cite[Theorem 1.2(3)]{FengShmerkin2014} and the uniform continuity on $(1,\infty)$ follows directly from \cite[Lemma 2.1]{KaenmakiVilppolainen2010}.
\end{proof}

The following lemma characterizes the continuity of the function $s \mapsto P(\A,s)$ at $0$.

\begin{lemma} \label{thm:pressure-right-continuous-at-zero}
  If $\A = (A_i)_{i \in J} \in M_2(\R)^J$ satisfies $\max_{i \in J} \|A_i\| < 1$, then the function $s \mapsto P(\A,s)$ is right-continuous at $0$ if and only if the semigroup $\{A_\iii : \iii \in J^*\}$ does not contain rank zero matrices.
\end{lemma}

\begin{proof}
  If the semigroup $\{A_\iii : \iii \in J^*\}$ does not contain rank zero matrices, then $\Sigma = J^\N$ and the right-continuity at $0$ is guaranteed by Lemma \ref{thm:pressure-continuous}. If $A_\iii$ has rank zero for some $\iii \in J^n$ and $n \in \N$, then clearly $\# \Sigma_n < \# J^n = (\# J)^n$. Fix $0 < s \le 1$ and notice that Lemma \ref{thm:pressure-continuous} implies
  \begin{equation*}
    P(\A,s) \le \frac{1}{n} \log \sum_{\iii \in \Sigma_n} \|A_\iii\|^s
  \end{equation*}
  and
  \begin{equation*}
    \lim_{s \downarrow 0} P(\A,s) \le \frac{1}{n} \log \# \Sigma_n < \frac{1}{n} \log \# J^n = P(\A,0),
  \end{equation*}
  where the limit exists by Lemma \ref{thm:pressure-continuous}. In particular, the function $s \mapsto P(\A,s)$ is not right-continuous at $0$.   
\end{proof}

The possible discontinuity at $1$ has already been observed by Feng and Shmerkin \cite[Remark 1.1]{FengShmerkin2014}. In their example, the pressure is not finite when $s > 1$, but it is easy to see that this is not a necessity. If $J = \{0,1\}$ and
\begin{equation*}
  A_0 =
  \begin{pmatrix}
    1 & 0 \\ 0 & 0
  \end{pmatrix}, \qquad
  A_1 =
  \begin{pmatrix}
    1 & 0 \\ 0 & 1
  \end{pmatrix},
\end{equation*}
then, by Lemma \ref{thm:pressure-continuous}, for $\A = (A_0,A_1) \in M_2(\R)^J$ we have $P(\A,1) = \log 2$ and $P(\A,s) = 0$ for all $s>1$. The continuity of the function $s \mapsto P(\A,s)$ at $1$ will be characterized for dominated and irreducible tuples in Lemma \ref{thm:pressure-continuous-at-one}.

Let us next determine when the pressure is finite. For that, we need the following definition. Given a tuple $\A = (A_i)_{i \in J} \in M_2(\R)^J$ of matrices, we define the \emph{joint spectral radius} by setting
\begin{equation*}
  \roo(\A) = \lim_{n \to \infty} \max_{\iii \in J^n} \|A_\iii\|^{1/n}.
\end{equation*}
As the operator norm is sub-multiplicative, the sequence $(\log \max_{\iii \in J^n} \|A_\iii\|)_{n \in \N}$ is sub-additive and hence, the limit above exists by Fekete's lemma.

\begin{lemma} \label{thm:joint-spectral-radius}
  If $\A = (A_i)_{i \in J} \in M_2(\R)^J$ is dominated or irreducible, then $\roo(\A) > 0$.
\end{lemma}

\begin{proof}
  Let us first assume that $\A$ is dominated and $\CC \subset \RP$ is a strongly invariant multicone for $\A$. Since there exists a multicone $\CC_0 \subset \RP$ such that $\bigcup_{\iii \in J^n} A_\iii \CC \subset \bigcup_{i \in J} A_i \CC \subset \CC_0 \subset \CC^o$ for all $n \in \N$, we find, by applying \cite[Lemma 2.2]{BochiMorris2015}, a constant $\kappa > 0$ such that
  \begin{equation} \label{eq:bochi-morris}
    \|A_\iii|V\| \ge \kappa \|A_\iii\|
  \end{equation}
  for all $V \in \CC_0$ and $\iii \in J^*$. It follows that if $V \in \CC_0$, then $A_\jjj V \in \CC_0$ and $\|A_\iii A_\jjj\| \ge \|A_\iii A_\jjj | V\| = \|A_\iii|A_\jjj V\| \|A_\jjj|V\| \ge \kappa^2\|A_\iii\|\|A_\jjj\|$ for all $\iii,\jjj \in J^*$. Therefore,
  \begin{equation*}
    \roo(\A) \ge \liminf_{n \to \infty} \max_{i_1 \cdots i_n \in J^n} \kappa^{2(n-1)/n} \|A_{i_1}\|^{1/n} \cdots \|A_{i_n}\|^{1/n} \ge \kappa^2 \min_{j \in J} \|A_j\| > 0
  \end{equation*}
  as claimed.

  Although the proof in the irreducible case can be found in \cite[Lemma 2.2]{Jungers2009}, we present the full details for the convenience of the reader. Denote the unit circle by $S^1$ and suppose that for each $k \in \N$ there is $x_k \in S^1$ such that for every $i \in J$ we have $|A_ix_k| < \frac{1}{k}$. By the compactness of $S^1$, there is $x \in S^1$ such that $|A_ix|=0$ for all $i \in J$. Choosing $V = \linspan(x) \in \RP$, we see that $A_iV = \{(0,0)\} \subset V$ for all $i \in J$ and $\A$ is reducible.

  It follows that there is $\delta > 0$ such that for every $x \in S^1$ there exists $i \in J$ for which $|A_ix| \ge \delta$. Let us next apply this inductively. Fix $x_0 \in S^1$ and choose $i_1 \in J$ such that $|A_{i_1}x_0| \ge \delta$. Write $x_1 = A_{i_1}x_0$ and choose $i_2 \in J$ such that $|A_{i_2}\frac{x_1}{|x_1|}| \ge \delta$ whence $|A_{i_2}A_{i_1}x_0| = |A_{i_2}x_1| \ge \delta|x_1| = \delta|A_{i_1}x_0| \ge \delta^2$. Continuing in this manner, we find for each $n \in \N$ a word $\iii_n \in J^n$ such that $\|A_{\iii_n}\| \ge |A_{\iii_n} x_0| \ge \delta^n$. Hence,
  \begin{equation*}
    \roo(\A) \ge \liminf_{n \to \infty} \|A_{\iii_n}\|^{1/n} \ge \delta > 0
  \end{equation*}
  as wished.
\end{proof}

The following two lemmas characterize the finiteness of the pressure.

\begin{lemma} \label{thm:pressure-finite1}
  If $\A = (A_i)_{i \in J} \in M_2(\R)^J$ satisfies $\max_{i \in J} \|A_i\| < 1$, then the following five conditions are equivalent:
  \begin{enumerate}
    \item\label{it:11} $P(\A,s) > -\infty$ for all $0 \le s \le 1$,
    \item\label{it:12} $\lim_{s \downarrow 0} P(\A,s) > -\infty$,
    \item\label{it:13} there does not exist $n \in \N$ such that $A_\iii = 0$ for all $\iii \in J^n$,
    \item\label{it:14} there exists $\jjj \in J^\N$ such that $A_{\jjj|_n} \ne 0$ for all $n \in \N$,
    \item\label{it:15} $\roo(\A)>0$.
  \end{enumerate}
  Furthermore, all of these conditions hold if $\A$ is dominated or irreducible.
\end{lemma}

\begin{proof}
  Notice that the limit in \eqref{it:12} exists by Lemma \ref{thm:pressure-continuous} and the implications \eqref{it:11} $\Rightarrow$ \eqref{it:12} and \eqref{it:14} $\Rightarrow$ \eqref{it:13} are trivial. Let us first show the implication \eqref{it:12} $\Rightarrow$ \eqref{it:13}. If \eqref{it:13} does not hold, then there exists $n_0 \in \N$ such that $A_\iii = 0$ for all $\iii \in J^{n_0}$. Since now $\|A_\iii\| = 0$ for all $\iii \in J^n$ and $n \ge n_0$, we see that $P(\A,s) = -\infty$ for all $s>0$ and \eqref{it:12} cannot hold.

  Let us then show the implication \eqref{it:13} $\Rightarrow$ \eqref{it:14}. If \eqref{it:14} does not hold, then for every $\jjj \in J^\N$ there is $n(\jjj) \in \N$ such that $A_{\jjj|_{n(\jjj)}} = 0$. By compactness of $J^\N$, there exist $M \in \N$ and $\jjj_1,\ldots,\jjj_M \in J^\N$ such that $\{[\jjj_i|_{n(\jjj_i)}]\}_{i=1}^M$ still covers $J^\N$. Choosing $n = \max_{i \in \{1,\ldots,M\}} n(\jjj_i)$, we see that for every $\iii \in J^n$ there is $i \in \{1,\ldots,M\}$ such that $A_\iii = A_{\jjj_i|_{n(\jjj_i)}}A_{\sigma^{n(\jjj_i)}\iii} = 0$ and \eqref{it:13} cannot hold.

  Since $\A$ is a tuple of strictly contractive matrices, the function $s \mapsto P(\A,s)$ is strictly decreasing whenever it is finite. Therefore, we have $P(\A,s) \ge P(\A,1) \ge \log \roo(\A)$ for all $0 \le s \le 1$ and hence, we have the implication \eqref{it:15} $\Rightarrow$ \eqref{it:11}. Therefore, to conclude the proof, it suffices to show the implication \eqref{it:13} $\Rightarrow$ \eqref{it:15} and also verify condition \eqref{it:15} when $\A$ is dominated or irreducible. While the latter is immediately assured by Lemma \ref{thm:joint-spectral-radius}, we also see that to prove the former, we may assume that $\A$ is reducible. This means that, after possibly a change of basis, the matrices $A_i$ in $\A$ are of the form
  \begin{equation*}
    A_i =
    \begin{pmatrix}
      a_i & b_i \\ 
      0 & c_i
    \end{pmatrix}
  \end{equation*}
  for all $i \in J$. Since $A_i(1,0) = a_i(1,0)$ and $A_i(\frac{b_i}{c_i-a_i},1) = c_i(\frac{b_i}{c_i-a_i},1)$ when $a_i \ne c_i$, we see that $\max\{|a_i|,|c_i|\} \le \|A_i\|$ for all $i \in J$. As the product of upper triangular matrices is upper triangular with diagonal entries obtained as products of the corresponding diagonal entries, we also have $\max\{|a_{i_1} \cdots a_{i_n}|, |c_{i_1} \cdots c_{i_n}|\} \le \|A_\iii\|$ for all $\iii = i_1 \cdots i_n \in J^n$ and $n \in \N$. Therefore, if condition \eqref{it:15} does not hold i.e.\ $\roo(\A) = 0$, then
  \begin{equation*}
    \max_{i \in J} |a_i| = \lim_{n \to \infty} \max_{i_1 \cdots i_n \in J^n} |a_{i_1} \cdots a_{i_n}|^{1/n} \le \roo(\A) = 0
  \end{equation*}
  and, similarly, $\max_{i \in J} |c_i| = 0$. In other words, the diagonal entries in all of the matrices $A_i$ are zero. Thus, $A_\iii = 0$ for all $\iii \in J^2$ and condition \eqref{it:13} does not hold.
\end{proof}

\begin{lemma} \label{thm:pressure-finite2}
  If $\A = (A_i)_{i \in J} \in M_2(\R)^J$ satisfies $\max_{i \in J} \|A_i\| < 1$, then the following five conditions are equivalent:
  \begin{enumerate}
    \item\label{it:20} $P(I^\N,\A,s) > -\infty$ for all $s \ge 0$,
    \item\label{it:21} $P(\A,s) > -\infty$ for all $s \ge 0$,
    \item\label{it:22} $\lim_{s \downarrow 1} P(\A,s) > -\infty$,
    \item\label{it:23} there does not exist $n \in \N$ such that $A_\iii$ has rank at most one for all $\iii \in J^n$,
    \item\label{it:24} there exists $j \in J$ such that $A_j \in GL_2(\R)$.
  \end{enumerate}
\end{lemma}

\begin{proof}
  Notice that the limit in \eqref{it:22} exists by Lemma \ref{thm:pressure-continuous} and the implications \eqref{it:20} $\Rightarrow$ \eqref{it:21} and \eqref{it:21} $\Rightarrow$ \eqref{it:22} are trivial. Let us first show the implication \eqref{it:22} $\Rightarrow$ \eqref{it:23}. If \eqref{it:23} does not hold, then there exists $n_0 \in \N$ such that $A_\iii$ has rank at most one for all $\iii \in J^{n_0}$. It follows that for every $\iii \in J^n$ and $n \ge n_0$ the rank of $A_\iii$ is at most one as it is bounded above by the rank of $A_{\iii|_{n_0}}$. Therefore, as $\fii^s(A_\iii) = 0$ for all $\iii \in J^n$, $n \ge n_0$, and $s>1$, we have $P(\A,s) = -\infty$ for all $s>1$ and \eqref{it:22} cannot hold.

  Let us then show the implication \eqref{it:23} $\Rightarrow$ \eqref{it:24}. If \eqref{it:24} does not hold, then $A_j$ has rank at most one for all $j \in J$. It follows that for every $\iii \in J^n$ and $n \in \N$ the rank of $A_\iii$ is at most one and \eqref{it:23} cannot hold.

  Finally, let us show the implication \eqref{it:24} $\Rightarrow$ \eqref{it:20}. The condition \eqref{it:24} implies that $A_{\jjj|_n} \in GL_2(\R)$ for all $n \in \N$ where $\jjj = jj\cdots \in J^\N$. Since $\fii^s(A_{\jjj|_n}) \ge \alpha_2(A_{\jjj|_n}) \ge \alpha_2(A_j)^n > 0$ for all $n \in \N$ and $s \ge 0$, we see that $P(I^\N,\A,s) \ge \log\alpha_2(A_j) > -\infty$ for all $s \ge 0$ as wished.
\end{proof}

\subsection{Equilibrium states}
Let $\MM_\sigma(J^\N)$ be the collection of all $\sigma$-invariant Borel probability measures on $J^\N$. If $0 < s \le 1$, then we say that a measure $\mu_K \in \MM_\sigma(J^\N)$ is \emph{$s$-Gibbs-type} if there exists a constant $C \ge 1$ such that
\begin{equation*}
  C^{-1}e^{-nP(\A,s)}\|A_\iii\|^s \le \mu_K([\iii]) \le Ce^{-nP(\A,s)}\|A_\iii\|^s
\end{equation*}
for all $\iii \in J^n$ and $n \in \N$.

\begin{lemma} \label{thm:weak-gibbs}
  If $\A = (A_i)_{i \in J} \in M_2(\R)^J$ satisfies $\max_{i \in J} \|A_i\| < 1$ and is dominated or irreducible, then for every $0 < s \le 1$ there exist a unique ergodic $s$-Gibbs-type measure $\mu_K \in \MM_\sigma(J^\N)$.
\end{lemma}

\begin{proof}
  Recall first that, by Lemma \ref{thm:pressure-finite1}, the pressure $P(\A,s)$ is finite for all $0 < s \le 1$. If $\A$ is irreducible, then the existence of the claimed measure $\mu_K \in \MM_\sigma(J^\N)$ follows immediately from \cite[Proposition 1.2]{FengKaenmaki2011}. We may thus assume that $\A$ is dominated. Fix $0 < s \le 1$ and notice that, by \eqref{eq:bochi-morris}, there exist $\kappa > 0$ and a multicone $\CC_0 \subset \RP$ such that $\|A_\iii|V\| \ge \kappa\|A_\iii\|$ for all $V \in \CC_0$ and $\iii \in J^*$. Fixing $V \in \CC_0$, we see that
  \begin{equation*}
    \log\|A_{\iii|_n}\|^s + \log\kappa^s \le \sum_{k=0}^{n-1} \log\|A_{\sigma^k \iii|_1}|A_{\sigma\iii}V\|^s \le \log\|A_{\iii|_n}\|^s
  \end{equation*}
  for all $\iii \in J^\N$ and $n \in \N$. By \cite[Theorems 1.7 and 1.16]{Bowen2008}, there exist an ergodic measure $\mu_K \in \MM_\sigma(J^\N)$ and a constant $C \ge 1$ such that
  \begin{equation*}
    \kappa^sC^{-1} e^{-nP(\A,s)}\|A_\iii\|^s \le \mu_K([\iii]) \le Ce^{-nP(\A,s)}\|A_\iii\|^s
  \end{equation*}
  for all $\iii \in J^n$ and $n \in \N$; see also \cite[Lemma 2.12]{BaranyKaenmakiYu2021}. The uniqueness of $\mu_K$ is now evident as two different ergodic measures are mutually singular.
\end{proof}

If $\A = (A_i)_{i \in J} \in M_2(\R)^J$ is dominated, then it follows from \eqref{eq:bochi-morris} that $\|A_\iii\| \ge \kappa^{2(n-1)}\|A_{i_1}\| \cdots \|A_{i_n}\| \ge \kappa^{2(n-1)}\min_{i \in J}\|A_i\|^n > 0$ for all $\iii = i_1 \cdots i_n \in J^n$ and $n \in \N$. Hence the semigroup $\{A_\iii : \iii \in J^*\}$ does not contain rank zero matrices and, by Lemma \ref{thm:pressure-right-continuous-at-zero}, the function $s \mapsto P(\A,s)$ is right-continuous at $0$. Furthermore, if there are no rank zero matrices, then $\Sigma = J^\N$ and the $s$-Gibbs-type measure $\mu_K \in \MM_\sigma(J^\N)$ is fully supported on $J^\N$. If $\A$ is irreducible, then $\mu_K$ is supported only on $\Sigma$.

Given $\mu \in \MM_\sigma(J^\N)$ and $\A = (A_i)_{i \in J} \in M_2(\R)^J$, we define for each $s \ge 0$ the \emph{energy} by setting
\begin{equation*}
  \Lambda(\mu,\A,s) = \lim_{n \to \infty} \frac{1}{n} \sum_{\iii \in J^n} \mu([\iii]) \log \fii^s(A_\iii) = \inf_{n \in \N} \frac{1}{n} \sum_{\iii \in J^n} \mu([\iii]) \log \fii^s(A_\iii).
\end{equation*}
The limit above exists or is $-\infty$ again by Fekete's lemma. Recall that the \emph{entropy} of $\mu$ is
\begin{equation*}
  h(\mu) = -\lim_{n \to \infty} \frac{1}{n} \sum_{\iii \in J^n} \mu([\iii]) \log\mu([\iii]).
\end{equation*}
It is well-known that
\begin{equation} \label{eq:eq-state-ineq}
  P(\A,s) \ge h(\mu) + \Lambda(\mu,\A,s)
\end{equation}
for all $\mu \in \MM_\sigma(J^\N)$ and $s \ge 0$; for example, see \cite[\S 3]{KaenmakiVilppolainen2010}. A measure $\mu_K \in \MM_\sigma(J^\N)$ is an \emph{$s$-equilibrium state} if it satisfies
\begin{equation} \label{eq:eq-state-def}
  P(\A,s) = h(\mu_K) + \Lambda(\mu_K,\A,s) > -\infty.
\end{equation}
The following lemma shows the uniqueness of the equilibrium state in dominated and irreducible cases.

\begin{lemma} \label{thm:unique-equilibrium-state}
  If $\A = (A_i)_{i \in J} \in M_2(\R)^J$ satisfies $\max_{i \in J} \|A_i\| < 1$ and is dominated or irreducible, then for every $0 < s \le 1$ the ergodic $s$-Gibbs-type measure $\mu_K \in \MM_\sigma(J^\N)$ is the unique $s$-equilibrium state.
\end{lemma}

\begin{proof}
  Fix $0 < s \le 1$ and let $\mu_K \in \MM_\sigma(J^\N)$ be the ergodic $s$-Gibbs-type measure. Since, by Lemmas \ref{thm:weak-gibbs} and \ref{thm:pressure-finite1},
  \begin{align*}
    h(\mu_K) + \Lambda(\mu_K,\A,s) &= \lim_{n \to \infty} \frac{1}{n} \sum_{\iii \in \Sigma_n} \mu_K([\iii]) \log \frac{\|A_\iii\|^s}{\mu_K([\iii])} \\ 
    &= \lim_{n \to \infty} \frac{1}{n} \sum_{\iii \in \Sigma_n} \mu_K([\iii]) \log e^{nP(\A,s)} = P(\A,s) > -\infty,
  \end{align*}
  we see that $\mu_K$ is an $s$-equilibrium state. As $\mu_K$ is ergodic, the uniqueness follows from \cite[Theorem 3.6]{KaenmakiVilppolainen2010}.
\end{proof}

If $\A = (A_i)_{i \in J} \in M_2(\R)^J$ contains an invertible matrix, then $I \ne \emptyset$ and, by Lemma \ref{thm:pressure-finite2}, $P(I^\N,\A,s) > -\infty$ for all $s \ge 0$. In this case, regardless of domination and irreducibility, it follows from \cite[Theorem 4.1]{Kaenmaki2004} that for every $s > 0$ there exists an ergodic measure $\nu_K \in \MM_\sigma(J^\N)$ supported on $I^\N$ such that
\begin{equation} \label{eq:equilibrium-state}
  P(I^\N,\A,s) = h(\nu_K) + \Lambda(\nu_K,\A,s).
\end{equation}
Note that such a measure is not necessarily unique; see \cite{FengKaenmaki2011,KaenmakiVilppolainen2010,KaenmakiMorris2018,BochiMorris2018}.

\begin{lemma} \label{thm:pressure-drop}
  If $\A = (A_i)_{i \in J} \in M_2(\R)^J$ satisfies $\max_{i \in J} \|A_i\| < 1$, contains a rank one matrix, and is dominated or irreducible, then
  \begin{equation*}
    P(I^\N,\A,s) < P(\A,s)
  \end{equation*}
  for all $0 \le s \le 1$.
\end{lemma}

\begin{proof}
  Since $\A$ is dominated or irreducible, Lemma \ref{thm:pressure-finite1} shows that $P(\A,s) > -\infty$ for all $0 \le s \le 1$. Notice that, by Lemma \ref{thm:pressure-continuous}, $P(I^\N,\A,0) = \log\# I < \log\# J = P(\A,0)$ and we may fix $0 < s \le 1$. Therefore, by Lemma \ref{thm:unique-equilibrium-state}, there exists unique $\mu_K \in \MM_\sigma(J^\N)$ such that
  \begin{equation} \label{eq:equilibrium-large}
    P(\A,s) = h(\mu_K) + \Lambda(\mu_K,\A,s) > -\infty.
  \end{equation}
  Furthermore, by Lemma \ref{thm:weak-gibbs}, $\mu_K$ satisfies
  \begin{equation*}
    \mu_K([\iii]) \ge C^{-1}e^{-nP(\A,s)} \|A_\iii\|^s > 0
  \end{equation*}
  for all $\iii \in \Sigma_n$ and $n \in \N$, where $C \ge 1$ is a constant. In particular, if $A_k$ is a rank one matrix in $\A$, then $\mu_K([k]) > 0$.

  If $\A$ does not contain invertible matrices, then trivially $P(I^\N,\A,s) = -\infty$ for all $s > 0$ and there is nothing to prove. We may thus assume that $\A$ contains an invertible matrix. Therefore, by \eqref{eq:equilibrium-state}, there exists a measure $\nu_K \in \MM_\sigma(J^\N)$ supported on $I^\N$ such that
  \begin{equation} \label{eq:equilibrium-small}
    P(I^\N,\A,s) = h(\nu_K) + \Lambda(\nu_K,\A,s).
  \end{equation}
  Since $A_k$ is not invertible and $\nu_K$ is supported on $I^\N$, we have $\nu_K([k]) = 0$. As $\mu_K$ is the unique measure in $\MM_\sigma(J^\N)$ satisfying \eqref{eq:equilibrium-large} and $\mu_K([k]) > \nu_K([k])$, we see that $\nu_K$ does not satisfy \eqref{eq:equilibrium-large} and therefore, by \eqref{eq:equilibrium-small},
  \begin{equation*}
    P(\A,s) > h(\nu_K) + \Lambda(\nu_K,\A,s) = P(I^\N,\A,s)
  \end{equation*}
  as claimed.
\end{proof}

The following lemma characterizes the continuity of the function $s \mapsto P(\A,s)$ at $1$.

\begin{lemma} \label{thm:pressure-continuous-at-one}
  If $\A = (A_i)_{i \in J} \in M_2(\R)^J$ satisfies $\max_{i \in J} \|A_i\| < 1$ and is dominated or irreducible, then the function $s \mapsto P(\A,s)$ is continuous at $1$ if and only if $\A$ does not contain rank one matrices.
\end{lemma}

\begin{proof}
  If $\A$ does not contain rank one matrices, then it contains only invertible or rank zero matrices. By Lemma \ref{thm:pressure-continuous}, rank zero matrices do not have any effect on the value of the pressure $P(\A,s)$ when $s > 0$. Therefore, rank zero matrices have no impact on the continuity at $1$ and we may assume that $\A \in GL_2(\R)^J$. But in this case, the continuity follows from \cite[Lemma 2.1]{KaenmakiVilppolainen2010}.

  Let us then assume that $\A$ contains a rank one matrix. If $\A$ does not contain invertible matrices, then, as the function $s \mapsto P(\A,s)$ is strictly decreasing, Lemma \ref{thm:pressure-finite2} implies that $P(\A,s) = -\infty$ for all $s > 1$. Furthermore, since $\A$ is dominated or irreducible, Lemma \ref{thm:pressure-finite1} shows that $P(\A,s) > -\infty$ for all $0 \le s \le 1$ and the function $s \mapsto P(\A,s)$ is discontinuous at $1$. We may thus assume that $\A$ contains an invertible matrix. By Lemma \ref{thm:pressure-finite2}, we thus have $P(I^\N,\A,s) > -\infty$ for all $s \ge 0$. Recall that, by \cite[Lemma 2.1]{KaenmakiVilppolainen2010}, the function $s \mapsto P(I^\N,\A,s)$ is continuous at $1$. Therefore, by Lemma \ref{thm:pressure-continuous}, showing
  \begin{equation*}
    P(I^\N,\A,1) < P(\A,1)
  \end{equation*}
  proves the function $s \mapsto P(\A,s)$ discontinuous at $1$. But, as $\A$ contains a rank one matrix, this follows immediately from Lemma \ref{thm:pressure-drop}.
\end{proof}

\section{Dimension of non-invertible self-affine sets} \label{sec:dim-results}

Recall that $J$ is a finite set and the affine iterated function system is a tuple $(f_i)_{i \in J}$ of contractive affine self-maps on $\R^2$ not having a common fixed point. We write $f_i = A_i+v_i$ for all $i \in J$, where $A_i \in M_2(\R)$ and $v_i \in \R^2$, and $f_\iii = f_{i_1} \circ \cdots \circ f_{i_n}$ for all $\iii = i_1 \cdots i_n \in J^n$ and $n \in \N$. We let $f_\varnothing = \mathrm{Id}$ to be the identity map. Note that the associated tuple of matrices $(A_i)_{i \in J}$ is an element of $M_2(\R)^J$ and satisfies $\max_{i \in J}\|A_i\|<1$.

If $I = \{i \in J : A_i \text{ is invertible}\}$ is non-empty, then the invertible self-affine set $X$ is associated to $(f_i)_{i \in I}$, and if $J \setminus I$ is non-empty, then the non-invertible self-affine set $X'$ is associated to $(f_i)_{i \in J}$. Recall the defining property \eqref{eq:self-affine-set-def} of a self-affine set. We use the convention that whenever we speak about a self-affine set, then it is automatically accompanied with a tuple of affine maps which defines it. This makes it possible to write that e.g.\ ``a non-invertible self-affine set is dominated'' which obviously then means that ``the associated tuple $\A = (A_i)_{i \in J}$ of matrices in $M_2(\R)^J$ is dominated''.

The study of non-invertible self-affine sets is connected to the theory of sub-self-affine sets. If the \emph{canonical projection} $\pi \colon J^\N \to \R^2$ is defined such that
\begin{equation*}
  \pi(\iii) = \lim_{n \to \infty} f_{\iii|_n}(0) = \lim_{n \to \infty} \sum_{k=1}^n A_{\iii|_{k-1}}v_{i_k}
\end{equation*}
for all $\iii = i_1i_2\cdots \in J^\N$, then we write $X'' = \pi(\Sigma)$, where $\Sigma = \{\iii \in J^\N : A_{\iii|_n}$ is non-zero for all $n \in \N\}$. Observe that $X = \pi(I^\N) \subset X'' \subset \pi(J^\N) = X'$ and, as $\sigma(\Sigma) \subset \Sigma$, the set $X''$ is \emph{sub-self-affine}, i.e.\
\begin{equation} \label{eq:sub-self-affine}
  X'' \subset \bigcup_{i \in J} f_i(X'');
\end{equation}
see \cite{KaenmakiVilppolainen2010}. The study is also connected to inhomogeneous self-affine sets. If $C \subset \R^2$ is compact, then there exists a unique non-empty compact set $X_C \subset \R^2$ such that
\begin{equation*}
  X_C = \bigcup_{i \in I} f_i(X_C) \cup C.
\end{equation*}
The set $X_C$ is called the \emph{inhomogeneous self-affine set} with condensation $C$. Such sets were introduced by Barnsley and Demko \cite{BarnsleyDemko1985} and they have been studied for example in \cite{Barnsley2006,KaenmakiLehrback2017,Burrell2019,BurrellFraser2020,BakerFraserMathe2019}. Note that $X_\emptyset$ is the invertible self-affine set $X$.

\begin{lemma} \label{thm:inhomog}
  If $X'$ and $X$ are non-invertible and invertible planar self-affine sets, respectively, and $X''$ is the associated sub-self-affine set defined in \eqref{eq:sub-self-affine}, then $X' \setminus X''$ is countable and
  \begin{equation*}
    X' = X_C = X \cup \bigcup_{\iii \in I^*} f_\iii(C),
  \end{equation*}
  where $X_C$ is the inhomogeneous self-affine set with condensation $C = \bigcup_{i \in J \setminus I} f_i(X')$.
\end{lemma}

\begin{proof}
  Let us first show that $X' \setminus X''$ is countable. Writing $v_\iii = \sum_{k=1}^n A_{\iii|_{k-1}}v_{i_k}$, we see that $f_\iii = A_\iii + v_\iii$ for all $\iii = i_1 \cdots i_n \in J^n$ and $n \in \N$. Let $\iii \in J^\N \setminus \Sigma$ and choose $n_0(\iii) = \min\{n \in \N : A_{\iii|_n}$ is zero$\}$. Since $v_{\iii|_{n+1}} = \sum_{k=1}^{n+1} A_{\iii|_{k-1}}v_{i_k} = A_{\iii|_n} v_{i_{n+1}} + \sum_{k=1}^n A_{\iii|_{k-1}}v_{i_k} = v_{\iii|_n}$ for all $n \ge n_0(\iii)$, a simple induction shows that
  \begin{equation*}
    f_{\iii|_n}(X') = \{v_{\iii|_{n_0(\iii)}}\}
  \end{equation*}
  for all $n \ge n_0(\iii)$. As $J^\N \setminus \Sigma$ is clearly separable, there exist countably many infinite words $\iii_1,\iii_2,\ldots \in J^\N \setminus \Sigma$ such that $J^\N \setminus \Sigma \subset \bigcup_{k \in \N} [\iii_k|_{n_0(\iii_k)}]$. It follows that
  \begin{equation*}
    X' \setminus X'' \subset \{v_{\iii_k|_{n_0(\iii_k)}} : k \in \N\}
  \end{equation*}
  is countable.

  Let us then prove the claimed equalities. Noting that the argument of \cite[Lemma 3.9]{Snigireva2008} works also in the self-affine setting, we have
  \begin{equation} \label{eq:inhomog-eq}
    X_C = X \cup \bigcup_{\iii \in I^*} f_\iii(C).
  \end{equation}
  To prove the remaining equality, let us first show that $X' \subset X_C$. To that end, fix $x \in X'$. By \eqref{eq:inhomog-eq}, we have $X \subset X_C$ and we may assume that $x \in X' \setminus X$. But this implies that there exist $\iii \in I^*$ and $i \in J \setminus I$ such that $x \in f_{\iii i}(X')$. Since, again by \eqref{eq:inhomog-eq},
  \begin{equation*}
    f_{\iii i}(X') \subset f_\iii(C) \subset \bigcup_{\iii \in I^*} f_\iii(C) \subset X_C
  \end{equation*}
  we have shown that $X' \subset X_C$. The inclusion $X_C \subset X'$ follows immediately from \eqref{eq:inhomog-eq} since we trivially have $X \subset X'$ and $f_\iii(C) \subset X'$ for all $\iii \in I^*$. Thus $X' = X_C$ as claimed.
\end{proof}

We are interested in the dimension of the non-invertible self-affine set. Relying on \eqref{eq:self-affine-set-def}, the non-invertible self-affine set $X'$ can naturally be covered by the sets $f_\iii(B)$, where $B$ is a ball containing $X'$. Note that such sets are ellipses or line segments, depending on whether the associated matrix is invertible or has rank one. Each set $f_\iii(B)$ can be covered by one ball of radius $\alpha_1(A_\iii)\diam(B)$ or by $\alpha_1(A_\iii)/\alpha_2(A_\iii)$ many balls of radius $\alpha_2(A_\iii)\diam(B)$. This motivates us to study the limiting behavior of sums $\sum_{\iii \in J^n} \fii^s(A_\iii)$ and hence, the pressure $P(\A,s)$.

Recall that the upper Minkowski dimension $\udimm$ is an upper bound for the Hausdorff dimension $\dimh$ for all compact sets; see \cite[\S 5.3]{Mattila1995}. The following lemma, generalizing \cite[Theorem 5.4]{Falconer1988}, shows that the affinity dimension is an upper bound for the upper Minkowski dimension for all non-invertible self-affine sets.

\begin{lemma} \label{thm:affinity-upper}
  If $X'$ is a planar self-affine set, then
  \begin{equation*}
    \udimm(X') \le \dimaff(\A).
  \end{equation*}
\end{lemma}

\begin{proof}
  We may assume that $\dimaff(\A) < 2$ as otherwise there is nothing to prove. Let $k \in \{0,1\}$ be such that $k \le \dimaff(\A) < k+1$. Fix $\dimaff(\A) < s < k+1$ and notice that $P(\A,s) < 0$. By \cite[Proposition 4.1]{Falconer1988}, we thus have
  \begin{equation} \label{eq:star-sum-finite}
    M = \sum_{\jjj \in J^*} \fii^s(A_\jjj) < \infty.
  \end{equation}
  Let $B$ be a ball containing $X'$. By scaling and translating, we may assume that $B$ is the unit ball. Write
  \begin{equation*}
    \CC_r = \{\iii \in J^* : \alpha_{k+1}(A_\iii) \le r < \alpha_{k+1}(A_{\iii^-})\}
  \end{equation*}
  for all $0<r<1$. If $\jjj \in J^\N$, then $\alpha_{k+1}(A_{\jjj|_0}) = \alpha_{k+1}(\mathrm{Id}) = 1$ and $\alpha_{k+1}(A_{\jjj|_n}) \to 0$ as $n \to \infty$. Therefore, for each $0<r<1$ there exists unique $n \in \N$ such that $\jjj|_n \in \CC_r$ and the collection $\{[\iii] : \iii \in \CC_r\}$ of pairwise disjoint cylinder sets is a cover of $J^\N$.

  Fix $0<r<1$ and $\iii \in \CC_r$, and observe that $f_\iii(B)$ is an ellipse with semi-axes $\alpha_1(A_\iii)$ and $\alpha_2(A_\iii)$. Since $\alpha_{k+1}(A_\iii) \le r < \alpha_{k+1}(A_{\iii^-})$, the set $f_\iii(B)$ is covered by
  \begin{equation*}
    \begin{cases}
      4, &\text{if } k=0, \\ 
      4\max\{r^{-1}\alpha_1(A_\iii),1\}, &\text{if } k=1
    \end{cases}
  \end{equation*}
  many balls of radius $r$. Notice that $\max\{r^{-1}\alpha_1(A_\iii),1\} \le r^{-1}\alpha_1(A_{\iii^-})$ and hence, $f_\iii(B)$ can be covered by $4\fii^k(A_{\iii^-})r^{-k}$ many balls of radius $r$. Write
  \begin{equation*}
    N_\iii(r) = 4\fii^k(A_{\iii^-})r^{-k}
  \end{equation*}
  and observe that
  \begin{equation*}
    N_\iii(r)r^s = 4\fii^k(A_{\iii^-})r^{s-k} \le 4\fii^k(A_{\iii^-})\alpha_{k+1}(A_{\iii^-})^{s-k} = 4\fii^s(A_{\iii^-})
  \end{equation*}
  for all $\iii \in \CC_r$. Recalling \eqref{eq:star-sum-finite}, we thus have
  \begin{equation} \label{eq:star-box}
  \begin{split}
    \sum_{\iii \in \CC_r} N_\iii(r) &\le 4r^{-s}\sum_{\iii \in \CC_r} \fii^s(A_{\iii^-}) = 4r^{-s}\sum_{\jjj \in J^*} \sum_{\iii \in \CC_r \,:\, \iii^- = \jjj} \fii^s(\A_{\jjj}) \\ 
    &\le 4r^{-s}\sum_{\jjj \in J^*} \# J \fii^s(A_\jjj) \le 4M\# Jr^{-s}.
  \end{split}
  \end{equation}
  Since $\{[\iii] : \iii \in \CC_r\}$ is a covering of $J^\N$, it follows that $\{f_\iii(B) : \iii \in \CC_r\}$ is a covering of $X'$. Hence $X'$ can be covered by $\sum_{\iii \in \CC_r} N_\iii(r)$ many balls of radius $r$. This together with \eqref{eq:star-box} gives $\udimm(X') \le s$. The proof is finished by letting $s \downarrow \dimaff(\A)$.
\end{proof}

It is easy to construct examples of self-affine sets having dimension strictly less than the affinity dimension. For example, several self-affine carpets have this property. Nevertheless, the classical result of Falconer \cite[Theorem 5.3]{Falconer1988} shows that, perhaps rather surprisingly, the Hausdorff dimension of a non-invertible self-affine set equals the affinity dimension for Lebesgue-almost every choice of translation vectors.

\begin{theorem} \label{thm:falconer}
  If $X_{\mathsf{v}}'$ is a planar self-affine set and $\A$ satisfies $\max_{i \in J}\|A_i\|<\frac12$, then
  \begin{equation*}
    \dimh(X_{\mathsf{v}}') = \min\{2,\dimaff(\A)\}
  \end{equation*}
  for $\LL^{2 \#J}$-almost all translation vectors $\mathsf{v} = (v_i)_{i \in J} \in (\R^2)^{\#J}$.
\end{theorem}

Originally, Falconer assumed that the matrices are invertible and their norms are bounded above by $\frac13$. Solomyak \cite{Solomyak1998} relaxed the bound to $\frac12$ which, by the example of Edgar \cite{Edgar1992}, is known to be the best possible. To see that $\min\{2,\dimaff(\A)\}$ in Theorem \ref{thm:falconer} is a lower bound for the Hausdorff dimension also when the matrices are non-invertible, by Lemma \ref{thm:pressure-continuous} it suffices to notice that \cite[Lemma 2.2]{Falconer1988} remains valid for all parameters $s$ strictly less than the rank of the matrix.

Recently a deterministic class of invertible self-affine sets were found for which the Hausdorff dimension equals the affinity dimension. We say that $X$ satisfies the \emph{open set condition} if there exists a non-empty open set $U \subset \R^2$ such that $f_i(U) \cap f_j(U) = \emptyset$ and $f_i(U) \subset U$ for all $i,j \in I$ with $i \ne j$. If such a set $U$ also intersects $X$, then we say that $X$ satisfies the \emph{strong open set condition}. The following breakthrough result for self-affine sets is proven by B\'ar\'any, Hochman, and Rapaport \cite[Theorems 1.1 and 7.1]{BHR}:

\begin{theorem} \label{thm:BHR}
  If $X$ is an invertible strictly affine strongly irreducible planar self-affine set satisfying the strong open set condition, then
  \begin{align*}
    \dimh(X) &= \min\{2,\dimaff(I^\N,\A)\}, \\ 
    \dimh(\proj_V(X)) &= \min\{1,\dimaff(I^\N,\A)\}
  \end{align*}
  for all $V \in \RP$.
\end{theorem}

We emphasize that Theorem \ref{thm:BHR} uses the assumption that the affine iterated function system consists only of invertible maps. It is currently not known whether the result holds also with non-invertible maps. We also remark that Hochman and Rapaport \cite{HochmanRapaport2021} have recently managed to relax the assumptions of the result. They showed that the strong open set condition can be replaced by exponential separation, a separation condition which allows overlapping.

The following three propositions collect our dimension results for non-invertible self-affine sets.

\begin{proposition} \label{thm:prop1}
  Suppose that $X'$ and $X$ are non-invertible and invertible planar self-affine sets, respectively. If
  \begin{align*}
    \dimh(X) &= \min\{2,\dimaff(I^\N,\A)\} \ge 1, \\ 
    \dimh(\proj_{V}(X)) &= \min\{1,\dimaff(I^\N,\A)\} = 1
  \end{align*}
  for all $V \in \RP$, then $\udimm(X') = \dimh(X)$ and $\dimh(\proj_V(X'))=1$ for all $V \in \RP$.
\end{proposition}

\begin{proof}
  To simplify notation, write $s = \dimaff(I^\N,\A)$. If $1 < s < \infty$, then Lemma \ref{thm:pressure-continuous} shows that $\dimaff(\A) = s \ge 1$. If $s=1$, then we get $P(\A,t) = P(I^\N,\A,t) < 0 = P(I^\N,\A,1) \le P(\A,1)$ for all $1<t<\infty$ and we again have $\dimaff(\A) = s \ge 1$. Therefore, by Lemma \ref{thm:affinity-upper}, we have $\udimm(X') \le \min\{2,\dimaff(\A)\} = \min\{2,s\} = \dimh(X) \le \dimh(X')$. To finish the proof, notice that $1 = \dimh(\proj_{V}(X)) \le \dimh(\proj_{V}(X')) \le 1$ for all $V \in \RP$.
\end{proof}

\begin{proposition} \label{thm:prop3}
  Suppose that $X'$ and $X$ are non-invertible and invertible planar self-affine sets, respectively. If $X'$ is dominated or irreducible, $\A$ contains a rank one matrix, $\dimaff(I^\N,\A) < 1$, and
  \begin{equation*}
    \dimh(X') = \min\{2,\dimaff(\A)\},
  \end{equation*}
  then $\dimh(X') > \udimm(X)$.
\end{proposition}

\begin{proof}
  To simplify notation, write $s=\dimaff(I^\N,\A)$. Since $s < 1$, Lemma \ref{thm:pressure-drop} implies that $0 = P(I^\N,\A,s) < P(\A,s)$. Therefore, as Lemmas \ref{thm:pressure-finite1} and \ref{thm:pressure-continuous} guarantee the continuity of the pressure, we have $s < \dimaff(\A)$. Therefore, by Lemma \ref{thm:affinity-upper}, we have $\udimm(X) \le s < \min\{2,\dimaff(\A)\} = \dimh(X')$.
\end{proof}

\begin{proposition} \label{thm:prop2}
  Suppose that $X'$ and $X$ are non-invertible and invertible planar self-affine sets, respectively. If $\A$ contains a rank one matrix and
  \begin{equation*}
    \dimh(X) = \dimh(\proj_V(X))  < 1
  \end{equation*}
  for all $V \in \RP$, then there exists a rank one matrix $A$ in $\A$ such that $\dimh(X') = \dimh(\proj_{\ker(A)^\bot}(X')) \le 1$.
\end{proposition}

\begin{proof}
  To simplify notation, write $s = \dimh(X)$. By Lemma \ref{thm:inhomog}, the non-invertible self-affine set can be expressed as an inhomogeneous self-affine set,
  \begin{equation*}
    X' = X_C = X \cup \bigcup_{\iii \in I^*} f_\iii(C),
  \end{equation*}
  where $C = \bigcup_{i \in J \setminus I} f_i(X')$. Therefore, by the countable stability of Hausdorff dimension,
  \begin{equation} \label{eq:main1}
  \begin{split}
    \dimh(X') &= \max\{s, \sup_{\iii \in I^*}\dimh(f_\iii(C))\} \\ 
    &= \max\{s, \dimh(C)\} = \max\{s, \max_{i \in J \setminus I}\dimh(A_i(X'))\}.
  \end{split}
  \end{equation}
  Let $A$ be a rank one matrix in $\A$ such that $\dimh(A(X')) = \max_{i \in J \setminus I}\dimh(A_i(X'))$. Since, by the assumption and Lemma \ref{thm:rank-one1}, $s = \dimh(\proj_{\ker(A)^\bot}(X)) \le \dimh(\proj_{\ker(A)^\bot}(X')) = \dimh(A(X'))$, the claim follows from \eqref{eq:main1}.
\end{proof}

We are now ready to prove the main result. The proof basically just applies Theorems \ref{thm:falconer} and \ref{thm:BHR} in the above propositions.

\begin{proof}[Proof of Theorem \ref{thm:main}]
  (1) Since, by Theorem \ref{thm:BHR}, we have
  \begin{align*}
    \dimh(X) &= \min\{2,\dimaff(I^\N,\A)\} \ge 1, \\
    \dimh(\proj_{V}(X)) &= \min\{1,\dimaff(I^\N,\A)\} = 1
  \end{align*}
  for all $V \in \RP$, Proposition \ref{thm:prop1} implies $\udimm(X') = \dimh(X)$ and $\dimh(\proj_V(X'))=1$ for all $V \in \RP$.

  (2) Since, by Theorem \ref{thm:falconer}, we have
  \begin{equation*}
    \dimh(X_{\mathsf{v}}') = \min\{2,\dimaff(\A)\}
  \end{equation*}
  for $\LL^{2 \#J}$-almost all $\mathsf{v} \in (\R^2)^{\#J}$, Proposition \ref{thm:prop3} implies $\dimh(X_{\mathsf{v}}') > \udimm(X_{\mathsf{v}})$ for $\LL^{2 \#J}$-almost all $\mathsf{v} \in (\R^2)^{\#J}$.

  (3) Since, by Theorem \ref{thm:BHR}, we have
  \begin{equation*}
    \dimh(X) = \dimh(\proj_V(X)) < 1
  \end{equation*}
  for all $V \in \RP$, Proposition \ref{thm:prop2} implies that there exists a rank one matrix $A$ in $\A$ such that $\dimh(X') = \dimh(\proj_{\ker(A)^\bot}(X')) \le 1$.
\end{proof}

\begin{ack}
  The authors thank De-Jun Feng for pointing out Lemma \ref{thm:affinity-upper}.
\end{ack}


\begin{thebibliography}{10}

\bibitem{BakerFraserMathe2019}
S.~Baker, J.~M. Fraser, and A.~M\'{a}th\'{e}.
\newblock Inhomogeneous self-similar sets with overlaps.
\newblock {\em Ergodic Theory Dynam. Systems}, 39(1):1--18, 2019.

\bibitem{BHR}
B.~B\'{a}r\'{a}ny, M.~Hochman, and A.~Rapaport.
\newblock Hausdorff dimension of planar self-affine sets and measures.
\newblock {\em Invent. Math.}, 216(3):601--659, 2019.

\bibitem{BaranyKaenmakiMorris2018}
B.~B\'{a}r\'{a}ny, A.~K\"{a}enm\"{a}ki, and I.~D. Morris.
\newblock Domination, almost additivity, and thermodynamic formalism for planar
  matrix cocycles.
\newblock {\em Israel J. Math.}, 239(1):173--214, 2020.

\bibitem{BaranyKaenmakiYu2021}
B.~B\'{a}r\'{a}ny, A.~K{\"a}enm{\"a}ki, and H.~Yu.
\newblock Finer geometry of planar self-affine sets.
\newblock Preprint, available at arXiv:2107.00983, 2021.

\bibitem{BaranyKort2024}
B.~B\'{a}r\'{a}ny and V.~K{\"o}rtv{\'e}lyesi.
\newblock On the dimension of planar self-affine sets with non-invertible maps.
\newblock Proc. Roy. Soc. Edinburgh Sect. A, to appear, available at
  arXiv:2302.13037, 2024.

\bibitem{Barnsley2006}
M.~F. Barnsley.
\newblock {\em Superfractals}.
\newblock Cambridge University Press, Cambridge, 2006.

\bibitem{BarnsleyDemko1985}
M.~F. Barnsley and S.~Demko.
\newblock Iterated function systems and the global construction of fractals.
\newblock {\em Proc. Roy. Soc. London Ser. A}, 399(1817):243--275, 1985.

\bibitem{BochiMorris2015}
J.~Bochi and I.~D. Morris.
\newblock Continuity properties of the lower spectral radius.
\newblock {\em Proc. Lond. Math. Soc. (3)}, 110(2):477--509, 2015.

\bibitem{BochiMorris2018}
J.~Bochi and I.~D. Morris.
\newblock Equilibrium states of generalised singular value potentials and
  applications to affine iterated function systems.
\newblock {\em Geom. Funct. Anal.}, 28(4):995--1028, 2018.

\bibitem{Bowen2008}
R.~Bowen.
\newblock {\em Equilibrium states and the ergodic theory of {A}nosov
  diffeomorphisms}, volume 470 of {\em Lecture Notes in Mathematics}.
\newblock Springer-Verlag, Berlin, revised edition, 2008.
\newblock With a preface by David Ruelle, Edited by Jean-Ren\'{e} Chazottes.

\bibitem{Burrell2019}
S.~A. Burrell.
\newblock On the dimension and measure of inhomogeneous attractors.
\newblock {\em Real Anal. Exchange}, 44(1):199--215, 2019.

\bibitem{BurrellFraser2020}
S.~A. Burrell and J.~M. Fraser.
\newblock The dimensions of inhomogeneous self-affine sets.
\newblock {\em Ann. Acad. Sci. Fenn. Math.}, 45(1):313--324, 2020.

\bibitem{Edgar1992}
G.~A. Edgar.
\newblock Fractal dimension of self-affine sets: some examples.
\newblock Number~28, pages 341--358. 1992.
\newblock Measure theory (Oberwolfach, 1990).

\bibitem{Falconer1988}
K.~J. Falconer.
\newblock The {H}ausdorff dimension of self-affine fractals.
\newblock {\em Math. Proc. Cambridge Philos. Soc.}, 103(2):339--350, 1988.

\bibitem{FengKaenmaki2011}
D.-J. Feng and A.~K{\"a}enm{\"a}ki.
\newblock Equilibrium states of the pressure function for products of matrices.
\newblock {\em Discrete Contin. Dyn. Syst.}, 30(3):699--708, 2011.

\bibitem{FengShmerkin2014}
D.-J. Feng and P.~Shmerkin.
\newblock Non-conformal repellers and the continuity of pressure for matrix
  cocycles.
\newblock {\em Geom. Funct. Anal.}, 24(4):1101--1128, 2014.

\bibitem{HochmanRapaport2021}
M.~Hochman and A.~Rapaport.
\newblock Hausdorff dimension of planar self-affine sets and measures with
  overlaps.
\newblock {\em J. Eur. Math. Soc. (JEMS)}, 24(7):2361--2441, 2022.

\bibitem{Hutchinson1981}
J.~E. Hutchinson.
\newblock Fractals and self-similarity.
\newblock {\em Indiana Univ. Math. J.}, 30(5):713--747, 1981.

\bibitem{Jungers2009}
R.~Jungers.
\newblock {\em The joint spectral radius}, volume 385 of {\em Lecture Notes in
  Control and Information Sciences}.
\newblock Springer-Verlag, Berlin, 2009.
\newblock Theory and applications.

\bibitem{Kaenmaki2004}
A.~K\"{a}enm\"{a}ki.
\newblock On natural invariant measures on generalised iterated function
  systems.
\newblock {\em Ann. Acad. Sci. Fenn. Math.}, 29(2):419--458, 2004.

\bibitem{KaenmakiLehrback2017}
A.~K\"{a}enm\"{a}ki and J.~Lehrb\"{a}ck.
\newblock Measures with predetermined regularity and inhomogeneous self-similar
  sets.
\newblock {\em Ark. Mat.}, 55(1):165--184, 2017.

\bibitem{KaenmakiMorris2018}
A.~K\"{a}enm\"{a}ki and I.~D. Morris.
\newblock Structure of equilibrium states on self-affine sets and strict
  monotonicity of affinity dimension.
\newblock {\em Proc. Lond. Math. Soc. (3)}, 116(4):929--956, 2018.

\bibitem{KaenmakiVilppolainen2010}
A.~K{\"a}enm{\"a}ki and M.~Vilppolainen.
\newblock Dimension and measures on sub-self-affine sets.
\newblock {\em Monatsh. Math.}, 161(3):271--293, 2010.

\bibitem{Marstrand1954}
J.~M. Marstrand.
\newblock Some fundamental geometrical properties of plane sets of fractional
  dimensions.
\newblock {\em Proc. London Math. Soc. (3)}, 4:257--302, 1954.

\bibitem{Mattila1995}
P.~Mattila.
\newblock {\em Geometry of Sets and Measures in Euclidean Spaces: Fractals and
  Rectifiability}.
\newblock Cambridge University Press, Cambridge, 1995.

\bibitem{Snigireva2008}
N.~Snigireva.
\newblock Inhomogeneous self-similar sets and measures.
\newblock PhD Dissertation, University of St Andrews, 2008.

\bibitem{Solomyak1998}
B.~Solomyak.
\newblock Measure and dimension for some fractal families.
\newblock {\em Math. Proc. Cambridge Philos. Soc.}, 124(3):531--546, 1998.

\end{thebibliography}

\end{document}